\documentclass[11pt]{amsart}

\usepackage{amsmath,amssymb,amsthm,mathtools}
\usepackage{microtype}
\usepackage[margin=1.12in]{geometry}
\usepackage[dvipsnames]{xcolor}
\usepackage[colorlinks=true,linkcolor=MidnightBlue,citecolor=MidnightBlue,urlcolor=MidnightBlue]{hyperref}
\usepackage[nameinlink,capitalise]{cleveref}

\allowdisplaybreaks
\numberwithin{equation}{section}

\newcommand{\R}{\mathbb{R}}
\newcommand{\Scal}{\operatorname{Scal}}
\newcommand{\Ric}{\operatorname{Ric}}
\newcommand{\tr}{\operatorname{tr}}
\newcommand{\diver}{\operatorname{div}}
\newcommand{\Vol}{\operatorname{Vol}}
\newcommand{\Id}{\operatorname{Id}}
\newcommand{\Acirc}{A^{\circ}}
\newcommand{\II}{\mathrm{II}}
\newcommand{\indm}{\operatorname{ind}_{-}}
\newcommand{\dmu}{\,d\mu}
\newcommand{\dV}{\,dV}

\newtheorem{theorem}{Theorem}[section]
\newtheorem{proposition}[theorem]{Proposition}
\newtheorem{lemma}[theorem]{Lemma}
\newtheorem{corollary}[theorem]{Corollary}
\newtheorem{remark}[theorem]{Remark}
\newtheorem{definition}[theorem]{Definition}
\newtheorem{conjecture}[theorem]{Conjecture}

\title[Sphere Theorems for Immersed Hypersurfaces with Constant Scalar Curvature]
{Sphere Theorems for Immersed Hypersurfaces with Constant Scalar Curvature} 

\author{Fagui Li}
\address{Frontier Interdisciplinary Domain, Beijing Institute of Technology, Zhuhai, Guangdong 519088, P. R. CHINA.}
\email{lifagui@bitzh.edu.cn}

\subjclass[2020]{53C42, 53C40, 53C20}
\keywords{Constant scalar curvature, immersed hypersurface,  two-convexity, Gromov  $k$-convexity,  Reilly formula}
\thanks{F. G. Li is partially supported by NSFC (No. 12271040 and 12501061), the Guangdong Provincial Association for Science and Technology Youth Talent Support Program (No. SKXRC2026413) and the Research Start-up Funding of Beijing Institute of Technology (No. 5640011253301).}
\hypersetup{
  pdftitle={Sphere theorems for immersed hypersurfaces with constant scalar curvature},
  pdfauthor={Fagui Li},
  pdfsubject={Rigidity of immersed hypersurfaces with constant scalar curvature},
  pdfkeywords={constant scalar curvature, immersed filling, Yau's conjecture, Reilly formula}
}

\begin{document}

\begin{abstract}
We prove sphere theorems for closed connected immersed hypersurfaces of constant scalar curvature in Euclidean space. Combining Ros's rigidity argument with the filling theorem of Huisken and Sinestrari, we show that every two-convex such hypersurface is a round sphere. As a consequence, Yau's conjecture holds for closed connected immersed hypersurfaces in \(\mathbb R^4\). Using Gromov's filling theorem, we also obtain higher-dimensional rigidity results under a principal-curvature index condition and an explicit scalar-curvature pinching condition.
\end{abstract}

\maketitle

\section{Introduction}
 Let
$
 F:M^n\looparrowright \R^{n+1}
$
be an immersion of a closed manifold, where
``closed'' means compact and without boundary.
Throughout the paper, unless otherwise stated, all manifolds are assumed
to be smooth and connected.
 Our detailed sign
conventions are fixed in \cref{sec:preliminaries}.  In particular, \(H\) denotes the trace, or unnormalized, mean curvature, and \(\sigma_2\) denotes the unnormalized second elementary symmetric function of the principal curvatures. In Euclidean space,
\[
 \Scal_M=2\sigma_2.
\]
We formulate Yau's problem \cite[Problem~31, p.~677]{Yau1982} in the
immersed setting considered here as the following conjecture.

\begin{conjecture}[Yau's conjecture \cite{Yau1982}]\label[conjecture]{conj:Yau-CSC}
Every closed immersed hypersurface in Euclidean space with constant scalar
curvature is isometric to a standard sphere.
\end{conjecture}

In dimension two, this conjecture follows from the classical
Hadamard--Liebmann theory; see
\cite[Section~1]{FonteneleNunez2018}.  For embedded hypersurfaces of
arbitrary dimension, Ros \cite{Ros1988} answered the question affirmatively; see also \cite{Ros1987,MontielRos1991} for the
corresponding higher-order mean-curvature framework.  Ros further
observed that his argument extends to an immersed hypersurface whenever
the boundary immersion extends across a compact manifold, after the
latter is endowed with the pullback Euclidean metric
\cite[Remark, p.~452]{Ros1987}.  A general immersion need not admit such
an extension or bound a well-defined Euclidean domain.  Rigidity was
previously obtained under
substantially stronger assumptions: nonnegative sectional curvature in
the work of Cheng--Yau \cite{ChengYau1977} and Ecker--Huisken
\cite{EckerHuisken1989}, nonnegative Ricci curvature in
\cite{Zheng1997}, local conformal flatness in \cite{Cheng2002}, and a
cohomogeneity-two hypothesis in \cite{Okayasu2005}; see also
Li \cite{Li1996} for related rigidity results for hypersurfaces with
constant scalar curvature in space forms.  In dimension
three, Cheng and Wan \cite{ChengWan1994} proved that a complete hypersurface in $\R^4$ with
nonzero constant mean curvature and constant scalar curvature is a
generalized cylinder $S^k(a)\times\R^{3-k}$, $k=1,2,3$; in particular, compactness leaves only the spherical
case.  More recently, Ge and Zhao \cite{GeZhao2026} proved that, for connected
hypersurfaces in real space forms, constancy of all Ricci eigenvalues is
equivalent to curvature homogeneity; in particular, a complete nonflat
Euclidean hypersurface of dimension $n\ge3$ with constant Ricci
eigenvalues is isoparametric.  These hypotheses are
stronger than constancy of the scalar curvature alone.  Fontenele and
N\'u\~nez  \cite{FonteneleNunez2018} proved a bounded complete rigidity theorem under the
additional assumption of constant mean curvature and recorded the
scalar-curvature-only problem as open.
To the best of our knowledge, Ros's observation has not previously
been combined with the filling theorems used here to obtain the
sphere theorems below. 
 The Cheng--Yau \cite{ChengYau1977} operator and its ellipticity
cone provide another basic viewpoint on constant scalar curvature.

The contribution of this paper is to combine Ros's pullback-metric
observation with two geometric sources of immersed fillings.  The
first is the theorem of Huisken and Sinestrari asserting that every
closed two-convex immersed hypersurface of dimension at least three
extends to an immersion of a compact handlebody
\cite[Corollary~1.2, pp.~138--139; see also
pp.~218--219]{HuiskenSinestrari2009}.  The second is Gromov's theorem
that a closed cooriented immersed hypersurface
$W^{N-1}\looparrowright\R^N$ bounds an immersed $N$-manifold if, for
the chosen coorientation, at least $k>N/2$ principal curvatures are
nonnegative \cite[p.~23]{Gromov1991}.  We refer to this as
\emph{Gromov's eigenvalue-count $k$-convexity}.  It is distinct from
the sum-type $p$-convexity used by Sha \cite{Sha1986,Sha1987} and from
the two-convexity appearing in the work of Fraser \cite{Fraser2002} and of
Huisken--Sinestrari \cite{HuiskenSinestrari2009}.  Our sign
convention agrees with Gromov's round-sphere normalization; this is
checked explicitly in \cref{sec:gromov}.  For a later treatment of
related mean-convexity and filling questions, see \cite{Gromov2014}.
The extension setup underlying the following definition already appears
in Ros \cite[Remark, p.~452]{Ros1987} and in Gromov
\cite[p.~23]{Gromov1991}.  In Ros's remark, the pullback-metric step is
attributed to U.~Pinkall.  We introduce the following terminology for
this geometric setup.

\begin{definition}\label[definition]{def:immersed-flat-filling}
An immersed filling of
$F:M^n\looparrowright\R^{n+1}$ consists of a compact
$(n+1)$-manifold $\Omega$ with boundary, a diffeomorphism
$\Phi:\partial\Omega\to M$, and an immersion
$
 G:\Omega\looparrowright\R^{n+1}
$
satisfying $G|_{\partial\Omega}=F\circ\Phi$.  We usually identify
$\partial\Omega$ with $M$ via $\Phi$.  If
$
 \bar g:=G^*\langle\cdot,\cdot\rangle,
$
then we call $(\Omega,\bar g,G,\Phi)$ an
immersed flat filling of $F$.
\end{definition}

Under either geometric hypothesis above, the corresponding filling theorem produces an
immersed filling in the sense of \cref{def:immersed-flat-filling}. The adjective ``flat'' is intrinsic:
on the interior $\Omega^\circ:=\Omega\setminus\partial\Omega$, the map $G$ is a local
diffeomorphism and, by the definition of $\bar g$, a local isometry. Hence
$\operatorname{Rm}_{\bar g}=0$ on $\Omega^\circ$, and therefore on all of $\Omega$
by smoothness.
The metric $\bar g$ induces the original
metric on the boundary, so Reilly's formula can be applied intrinsically
on $(\Omega,\bar g)$ even when $G$ is not injective.

The following statement is the $r=2$ specialization of Ros's
argument, together with the extension noted in
\cite[Theorem~2 and the Remark on p.~452]{Ros1987}.
We record it in the terminology introduced above and include the
details needed for our application.

\begin{theorem}[Ros \cite{Ros1987}]\label{thm:flat-filling}
Let $n\ge2$, and let
$
F:M^n\looparrowright\R^{n+1}
$
be a closed immersed hypersurface with constant
scalar curvature. If $F$ admits an immersed filling, then $F$ is a
diffeomorphism onto a round sphere.
\end{theorem}

Combining \cref{thm:flat-filling} with the filling theorem of
Huisken--Sinestrari \cite{HuiskenSinestrari2009} gives our main sphere theorem.

\begin{theorem}\label{thm:two-convex}
Let $n\ge3$, and let $F:M^n\looparrowright\R^{n+1}$ be a closed two-convex immersed hypersurface.  If the scalar curvature of $M$ is constant, then $F$ is a diffeomorphism onto a round sphere.
\end{theorem}

Here and below, two-convexity is understood in the weak sense used by
Huisken and Sinestrari: after choosing a mean-convex unit normal and
ordering the principal curvatures as
$\lambda_1\le\cdots\le\lambda_n$, one has
\[
 \lambda_1+\lambda_2\ge0.
\]
This is a \emph{sum-type} convexity condition; compare Fraser's strict
two-convexity and Sha's $p$-convexity \cite{Fraser2002,Sha1986,Sha1987}.
It is not the same notion as Gromov's eigenvalue-count $k$-convexity
used in \cref{sec:gromov}.  For a constant-scalar-curvature immersion,
the orientation with $H>0$ will be constructed intrinsically in
\cref{lem:positivity}.  In dimension three the curvature equation then
implies strict two-convexity.

\begin{corollary}\label[corollary]{cor:dimension-three}
Every closed immersed hypersurface
$
 F:M^3\looparrowright\R^4
$
with constant scalar curvature is a diffeomorphism onto a round $3$-sphere. In particular, Yau's conjecture holds in dimension three.
\end{corollary}

There is also a genuinely high-dimensional consequence.  For the
orientation with $H>0$, let
\[
 \indm(A_p)
 :=\#\{i:\lambda_i(p)<0\}
\]
denote the negative inertia index of the shape operator, counted with
multiplicity.
\begin{theorem}\label{thm:negative-index}
Let $n\ge4$, and let
$F:M^n\looparrowright\R^{n+1}$ be a closed immersed
hypersurface with constant scalar curvature.  Let $\nu$ be the global
unit normal furnished by \cref{lem:positivity}, chosen so that $H>0$, and let
$A=A_\nu$.  If
\begin{equation}\label{eq:negative-index-hypothesis}
 \indm(A_p)
 \le \left\lfloor\frac{n-2}{2}\right\rfloor
 \qquad\text{for every }p\in M,
\end{equation}
then $F$ is a diffeomorphism onto a round sphere.
\end{theorem}

\begin{corollary}\label{cor:scalar-pinching}
Let $n\ge4$, and let
$F:M^n\looparrowright\R^{n+1}$ be a closed immersed
hypersurface with constant scalar curvature.  Set
$
 m=\left\lfloor\frac{n+1}{2}\right\rfloor.
$
If
\begin{equation}\label{eq:scalar-pinching}
 H^2\le \frac{2m}{m-1}\,\sigma_2
 \qquad\text{on }M,
\end{equation}
then $F$ is a diffeomorphism onto a round sphere.  Equivalently, since
$\Scal_M=2\sigma_2$, it is enough to assume
\[
 H^2\le \frac{m}{m-1}\,\Scal_M.
\]
\end{corollary}

For completeness, we recall Ros's argument in the case $r=2$ and in
our trace-mean-curvature convention.  We first show that
the constant second elementary symmetric curvature $\sigma_2$ is
positive and choose the global normal so that the trace mean curvature
$H$ is positive.  The first Newton transformation gives
\[
 (n-1)\int_M H\,d\mu=2\sigma_2\int_M u\,d\mu,
\]
where $u$ is the support function.  The filling identifies the latter
integral, with the correct sign, with
$(n+1)\Vol_{\bar g}(\Omega)$.  Reilly's formula then gives
\[
 \int_M\frac1H\,d\mu\ge \frac{n+1}{n}\Vol_{\bar g}(\Omega).
\]
On the other hand, the pointwise Newton inequality gives
\[
 \frac{H}{\sigma_2}\ge \frac{2n}{(n-1)H}.
\]
The constants match exactly, forcing equality and hence total
umbilicity.  Two-convexity and the negative-inertia hypothesis enter
only through their respective filling theorems; the scalar pinching
condition is reduced algebraically to the negative-inertia condition.

The paper is organized as follows.  In \cref{sec:preliminaries} we fix
the sign conventions and collect the basic identities and positivity
properties used throughout the paper.  Ros's rigidity argument on an
immersed flat filling, including the boundary orientation and
normalization details, is recalled in \cref{sec:flat-filling}.  In
\cref{sec:two-convex} we apply the Huisken--Sinestrari filling theorem
and prove \cref{cor:dimension-three}.  In \cref{sec:gromov} we use
Gromov's immersed filling theorem to prove \cref{thm:negative-index}
and \cref{cor:scalar-pinching}.  The final subsection isolates the
remaining higher-dimensional obstruction.

\section{Preliminaries and conventions}\label{sec:preliminaries}

Throughout, $D$ denotes the Euclidean connection and $\nabla$ the
Levi-Civita connection of the metric induced on $M$.  As usual, in
expressions involving $D$ we identify a tangent vector $X\in TM$ with
$dF(X)$.  On a filling we write $\bar\nabla$, $\bar\Delta$, and $dV$
for the corresponding connection, Laplacian, and volume measure;
$d\mu$ denotes the hypersurface measure.

If $\nu$ is a locally defined unit normal, the shape operator $A$ and
the scalar second fundamental form $\II$ are defined by
\[
 D_X\nu=A X,
 \qquad
 \II(X,Y)=\langle AX,Y\rangle.
\]
Thus
\[
 D_XdF(Y)=dF(\nabla_XY)-\II(X,Y)\nu,
\]
and a round sphere has positive principal curvatures with respect to
its outward normal.  The eigenvalues of $A$ are denoted by
$\lambda_1,\ldots,\lambda_n$ and are called the principal curvatures;
whenever they are ordered, we use
$\lambda_1\le\cdots\le\lambda_n$.  We use the unnormalized elementary
symmetric functions
\[
 H=\sigma_1=\tr A,
 \qquad
 \sigma_2=\sum_{1\le i<j\le n}\lambda_i\lambda_j,
 \qquad
 |A|^2=\sum_{i=1}^n\lambda_i^2.
\]
Thus $H$ is the trace, rather than the normalized, mean curvature.
The Euclidean Gauss equation gives
\begin{equation}\label{eq:gauss-scalar}
 \Scal_M=H^2-|A|^2=2\sigma_2.
\end{equation}
We also set
\[
 P_1=H\Id-A,
 \qquad
 \Acirc=A-\frac{H}{n}\Id.
\]
The Codazzi equation implies
\begin{equation}\label{eq:P1-divergence}
 \diver_M P_1=0,
\end{equation}
and direct traces give
\begin{equation}\label{eq:P1-traces}
 \tr P_1=(n-1)H,
 \qquad
 \tr(P_1A)=H^2-|A|^2=2\sigma_2.
\end{equation}
These are the standard identities underlying the Cheng--Yau operator
\cite{ChengYau1977}.

We will also use Reilly's formula with the same sign convention.  Let
$(\Omega^{n+1},\bar g)$ be compact with smooth boundary $M$, let
$\eta$ be the outward unit normal, and put
$\bar A X=\bar\nabla_X\eta$ and $\bar H=\tr_M\bar A$.  If
$f|_M=0$, Reilly's formula
\cite[Theorem~1 and formula~(14), p.~463]{Reilly1977} reduces to
\begin{equation}\label{eq:reilly-preliminary}
 \int_\Omega\left((\bar\Delta f)^2
 -|\bar\nabla^2f|^2
 -\Ric_{\bar g}(\bar\nabla f,\bar\nabla f)\right)\dV
 =\int_M\bar H f_\eta^2\dmu,
\end{equation}
where $\Ric_{\bar g}$ denotes the Ricci tensor of $\bar g$ and
$f_\eta=\langle\bar\nabla f,\eta\rangle$.  This displayed form also
fixes all boundary signs used below.

\subsection{Positivity and a Minkowski identity}\label{sec:positivity}
We first establish two elementary facts that will be used repeatedly below:
the canonical choice of global normal with $H>0$, together with the positivity
of $P_1$, and the first Minkowski identity.
\begin{lemma}\label[lemma]{lem:positivity}
Let $n\ge2$, let $F:M^n\looparrowright\R^{n+1}$ be closed, and
suppose that $\sigma_2$ is constant.  Then
$\sigma_2>0$.  Moreover, the normal bundle is trivial and the global
unit normal can be chosen so that
\[
 H>0\quad\text{on }M.
\]
For this choice, $P_1$ is positive definite everywhere.
\end{lemma}

\begin{proof}
Choose $a\in\R^{n+1}\setminus F(M)$ and let
\[
 \rho=\frac12|F-a|^2.
\]
At a maximum point $p$ of $\rho$, the vector $F(p)-a$ is normal to
$M$.  Since $a\notin F(M)$, we may choose a local unit normal near
$p$ so that
\[
 F(p)-a=r\nu(p),\qquad r=|F(p)-a|>0.
\]
At $p$, for every $X\in T_pM$, the Gauss formula gives
\[
 \nabla^2\rho(X,X)=|X|^2-r\langle AX,X\rangle.
\]
Since $\nabla^2\rho\le0$ at $p$,
\[
 \langle AX,X\rangle\ge r^{-1}|X|^2.
\]
Thus all principal curvatures at $p$ are positive, and consequently
\[
 \sigma_2=\sigma_2(p)>0,
\]
where $n\ge2$ is used in the last implication.

By \eqref{eq:gauss-scalar},
\begin{equation}\label{eq:H-nonzero}
 H^2=|A|^2+2\sigma_2>0.
\end{equation}
At this stage $H$ is only locally defined, but $H^2$ is globally
well defined: reversing a local unit normal sends $A$ to $-A$ and
$H$ to $-H$, while $|A|^2$ and $\sigma_2$ are unchanged.
To make the orientability argument explicit, take local unit normals
$\nu_\alpha$ and write $H_\alpha=\tr A_{\nu_\alpha}$.  On an overlap
where $\nu_\beta=-\nu_\alpha$, one has
$H_\beta=-H_\alpha$.  Hence the local fields
$H_\alpha\nu_\alpha$ patch together to a global normal field
$\mathcal H$.  By \eqref{eq:H-nonzero}, $\mathcal H$ never vanishes.
Thus
\[
 \nu=\frac{\mathcal H}{|\mathcal H|}
\]
is a global unit normal, and the corresponding scalar mean curvature
is $H=|\mathcal H|>0$.

Finally, for every principal curvature $\lambda_i$,
\[
 H^2=|A|^2+2\sigma_2>|A|^2\ge\lambda_i^2.
\]
Since $H>0$, it follows that $H>|\lambda_i|$, and hence every eigenvalue $H-\lambda_i$ of $P_1$ is positive.
\end{proof}

\begin{remark}\label[remark]{rem:n-minus-one-convex}
If $\lambda_1\le\cdots\le\lambda_n$, positivity of $P_1$ gives
\[
 \lambda_1+\cdots+\lambda_{n-1}=H-\lambda_n>0.
\]
Thus every closed constant-scalar-curvature immersion is strictly
$(n-1)$-convex in this sense.  Precisely when $n=3$, this is strict
two-convexity.
\end{remark}

The next identity is the first Hsiung--Minkowski formula
\cite{Hsiung1954,MontielRos1991}; we include its short derivation to fix
the normalizations.

\begin{lemma}\cite{Hsiung1954,MontielRos1991}\label[lemma]{lem:minkowski}
Under the hypotheses and orientation of \cref{lem:positivity}, fix $a\in\R^{n+1}$ and define
\[
 u=\langle F-a,\nu\rangle.
\]
Then
\begin{equation}\label{eq:minkowski}
 (n-1)\int_M H\,d\mu
 =2\sigma_2\int_M u\,d\mu.
\end{equation}
In particular,
\begin{equation}\label{eq:u-positive}
 \int_M u\,d\mu>0.
\end{equation}
\end{lemma}

\begin{proof}
Let
\[
 Z=(F-a)^T=F-a-u\nu.
\]
Moreover, for $X\in TM$,
\[
 \nabla_X Z=X-uAX.
\]
Using \eqref{eq:P1-divergence} and \eqref{eq:P1-traces}, we obtain
\begin{align*}
 \diver_M(P_1Z)
 &=\tr(P_1)-u\tr(P_1A)\\
 &=(n-1)H-u(H^2-|A|^2)\\
 &=(n-1)H-2\sigma_2u.
\end{align*}
Integrating over the closed manifold $M$ proves \eqref{eq:minkowski}.
The positivity in \eqref{eq:u-positive} follows from $H>0$ and
$\sigma_2>0$.
\end{proof}

\section{Ros's rigidity argument on an immersed flat filling}
\label{sec:flat-filling}

This section gives the $r=2$ specialization of the argument in
\cite[Theorem~2 and the Remark on p.~452]{Ros1987}.  We include the
details because the boundary orientation and the trace convention for
the mean curvature are important below.

Throughout this section, assume that
$F:M^n\looparrowright\R^{n+1}$ satisfies the hypotheses of
\cref{thm:flat-filling}.  Let $(\Omega,G,\Phi)$ be an immersed filling
and identify $\partial\Omega$ with $M$ using $\Phi$.  Thus
\[
 G:\Omega\looparrowright\R^{n+1}
\]
restricts to $F$ on the boundary.  Equip $\Omega$ with
\[
 \bar g=G^*\langle\cdot,\cdot\rangle.
\]
On $\Omega^\circ$, the map $G$ is an immersion between manifolds of the
same dimension and hence a local diffeomorphism.  By the definition of
$\bar g$, it is also a local isometry to Euclidean space.  Consequently,
$\overline{\operatorname{Rm}}=0$ in the interior and, by smoothness up
to the boundary,
\begin{equation}\label{eq:flat}
 \overline{\operatorname{Rm}}\equiv0.
\end{equation}
The boundary metric and measure induced by $\bar g$ agree with those
induced by $F$.  Let $\eta$ denote the outward unit normal to
$M=\partial\Omega$ with respect to $\bar g$, and write
\[
 V=\Vol_{\bar g}(\Omega).
\]
Fix $a\in\R^{n+1}$ and set
\[
 u=\langle F-a,\nu\rangle.
\]

We first spell out the boundary-orientation compatibility implicit in
Ros's remark.  This ensures that the positive mean curvature selected
in \cref{lem:positivity} is exactly the boundary mean curvature required
in Reilly's formula.

\begin{lemma}\label[lemma]{lem:normal-compatibility}
Along $M$ one has
\[
 dG(\eta)=\nu.
\]
Moreover,
\begin{equation}\label{eq:support-volume}
 \int_M u\,d\mu=(n+1)V.
\end{equation}
If $\bar A X=\bar\nabla_X\eta$ is the boundary shape operator of
$(\Omega,\bar g)$, then, under the identification by $dF=dG|_{TM}$,
\[
 \bar A=A.
\]
In particular, the outward boundary mean curvature equals $H>0$.
\end{lemma}

\begin{proof}
Both $dG(\eta)$ and $\nu$ are unit normal fields along the connected
immersed hypersurface $F$.  Hence
\[
 dG(\eta)=\varepsilon\nu
\]
for a constant $\varepsilon\in\{-1,1\}$.

Let $\mathcal R_a(x)=x-a$ be the Euclidean radial vector field.  Since
$dG$ is an isomorphism, there is a unique smooth vector field $Y$ on
$\Omega$ defined by
\[
 dG(Y)=\mathcal R_a\circ G=G-a.
\]
Since $G$ is a local isometry, for every vector field $X$ on $\Omega$,
\[
 dG(\bar\nabla_XY)
 =D_{dG(X)}\mathcal R_a
 =dG(X).
\]
Thus
\[
 \bar\nabla Y=\Id,
 \qquad
 \diver_{\bar g}Y=n+1.
\]
Moreover, $G^*(dx^1\wedge\cdots\wedge dx^{n+1})$ is nowhere zero, so
$\Omega$ is orientable.  The divergence theorem gives
\begin{align*}
 (n+1)V
 &=\int_M\langle Y,\eta\rangle_{\bar g}\,d\mu\\
 &=\int_M\langle G-a,dG(\eta)\rangle\,d\mu\\
 &=\varepsilon\int_M u\,d\mu.
\end{align*}
By \eqref{eq:u-positive}, $\int_M u\,d\mu>0$.  Since $V>0$, we must
have $\varepsilon=1$, and \eqref{eq:support-volume} follows.

For $X\in TM$, compatibility of the Levi-Civita connections under the
local isometry $G$ now gives
\begin{align*}
 dG(\bar A X)
 &=dG(\bar\nabla_X\eta)\\
 &=D_{dF(X)}dG(\eta)\\
 &=D_{dF(X)}\nu
 =dF(AX).
\end{align*}
Since $dG=dF$ on $TM$, this proves $\bar A=A$, and hence the outward
boundary mean curvature is $H>0$.
\end{proof}

Combining \cref{lem:minkowski,lem:normal-compatibility} yields the exact identity
\begin{equation}\label{eq:exact-H-over-sigma2}
 \int_M\frac{H}{\sigma_2}\,d\mu
 =\frac{2(n+1)}{n-1}V.
\end{equation}

We next record the flat specialization of Ros's Reilly-formula version
of the classical Heintze--Karcher inequality
\cite[Theorem~1, pp.~449--450]{Ros1987}; see also
\cite{HeintzeKarcher1978}.  Ros uses the normalized mean curvature
$H_1=H/n$, whereas $H$ below denotes the trace mean curvature.  We
include the short proof to make this normalization and all boundary
signs explicit.

\begin{proposition}[Heintze--Karcher--Ros inequality \cite{HeintzeKarcher1978,Ros1987}]\label[proposition]{prop:HK}
Let $(\Omega^{n+1},\bar g)$ be a compact flat Riemannian manifold with
smooth boundary $M$.  Let $\eta$ be the outward unit normal, set
$\bar A X=\bar\nabla_X\eta$, and let $H=\tr_M\bar A$.  If $H>0$, then
\begin{equation}\label{eq:HK}
 \int_M\frac1H\,d\mu
 \ge \frac{n+1}{n}\Vol_{\bar g}(\Omega).
\end{equation}
\end{proposition}

\begin{proof}
We use $\bar\Delta=\diver_{\bar g}\bar\nabla$.  Standard Dirichlet
theory gives a unique smooth solution of
\begin{equation}\label{eq:dirichlet}
 \begin{cases}
  \bar\Delta f=1&\text{in }\Omega,\\
  f=0&\text{on }M.
 \end{cases}
\end{equation}
Since $\bar g$ is flat and $f$ vanishes on the boundary,
\eqref{eq:reilly-preliminary} reduces to
\begin{equation}\label{eq:reilly-reduced}
 \int_M Hf_\eta^2\,d\mu
 =\int_\Omega\bigl((\bar\Delta f)^2-|\bar\nabla^2f|^2\bigr)\,dV.
\end{equation}
The trace inequality in dimension $n+1$ gives
\[
 |\bar\nabla^2f|^2\ge\frac{(\bar\Delta f)^2}{n+1}=\frac{1}{n+1}.
\]
Consequently,
\begin{equation}\label{eq:reilly-upper}
 \int_M Hf_\eta^2\,d\mu
 \le\frac{n}{n+1}V.
\end{equation}
On the other hand, the divergence theorem and \eqref{eq:dirichlet} imply
\[
 \int_M f_\eta\,d\mu=V.
\]
Weighted Cauchy--Schwarz and \eqref{eq:reilly-upper} therefore yield
\begin{align*}
 V^2
 &=\left(\int_M f_\eta\,d\mu\right)^2\\
 &\le \left(\int_M Hf_\eta^2\,d\mu\right)
       \left(\int_M\frac1H\,d\mu\right)\\
 &\le \frac{n}{n+1}V\int_M\frac1H\,d\mu.
\end{align*}
Dividing by $V>0$ proves \eqref{eq:HK}.
\end{proof}

For completeness, we now present Ros's argument in the case \(r=2\), writing the Newton inequality in the exact defect form needed to identify the equality case.

\begin{proof}[\textbf{Proof of \cref{thm:flat-filling}}]
By \cref{lem:positivity}, $\sigma_2>0$ and the normal is chosen so that
$H>0$.  The Gauss equation gives
\[
 2\sigma_2=H^2-|A|^2.
\]
A direct calculation gives the exact Newton-defect identity
\begin{equation}\label{eq:newton-defect}
 \frac{H}{\sigma_2}-\frac{2n}{(n-1)H}
 =\frac{n|\Acirc|^2}{(n-1)\sigma_2H}\ge0.
\end{equation}
Indeed,
\[
 (n-1)H^2-2n\sigma_2
 =n|A|^2-H^2
 =n|\Acirc|^2.
\]
Using \eqref{eq:exact-H-over-sigma2}, \eqref{eq:newton-defect}, and
\cref{prop:HK}, we obtain
\begin{align}
 \frac{2(n+1)}{n-1}V
 &=\int_M\frac{H}{\sigma_2}\,d\mu \notag\\
 &\ge\frac{2n}{n-1}\int_M\frac1H\,d\mu \notag\\
 &\ge\frac{2n}{n-1}\frac{n+1}{n}V
 =\frac{2(n+1)}{n-1}V.
 \label{eq:equality-chain}
\end{align}
Thus equality holds throughout.  In particular, integrating \eqref{eq:newton-defect} gives
\[
 \int_M\frac{n|\Acirc|^2}{(n-1)\sigma_2H}\,d\mu=0.
\]
Since $\sigma_2H>0$, we conclude that
\[
 \Acirc\equiv0.
\]
Hence $A=\kappa\Id$, where $\kappa=H/n>0$.  Since
\[
 \sigma_2=\binom{n}{2}\kappa^2
\]
is constant, $\kappa$ is constant.  Therefore
\[
 D_X\left(F-\frac1\kappa\nu\right)
 =X-\frac1\kappa AX=0
\]
for every $X\in TM$.  There is a point $q\in\R^{n+1}$ such that
\[
 F-\frac1\kappa\nu=q,
\]
and hence
\[
 |F-q|=\frac1\kappa.
\]
Thus $F(M)$ lies in the round sphere $S^n(q,\kappa^{-1})$.  The map
\[
 F:M\longrightarrow S^n(q,\kappa^{-1})
\]
is a local diffeomorphism.  Its image is open; it is also compact and
hence closed in the connected sphere.  Therefore the image is the
whole sphere.  Since $M$ is compact, $F$ is proper, so this surjective
local diffeomorphism is a covering map.  For $n\ge2$, the target sphere
is simply connected.  Because $M$ is connected, the covering has one
sheet.  Thus $F$ is a diffeomorphism onto the round sphere.
\end{proof}

\section{Two-convex hypersurfaces and the three-dimensional case}
\label{sec:two-convex}

We now invoke the following consequence of the Huisken--Sinestrari mean curvature flow with surgery, which provides precisely the immersed extension required in Ros's remark.

\begin{theorem}[Huisken--Sinestrari \cite{HuiskenSinestrari2009}]\label{thm:HS}
Let $n\ge3$.  Every smooth closed two-convex immersion
\[
 F:M^n\looparrowright\R^{n+1}
\]
admits an immersed filling by a compact handlebody.
\end{theorem}

More precisely, this extension statement is contained in the immersed
case of \cite[Corollary~1.2 and the last paragraph of its
proof]{HuiskenSinestrari2009}: one obtains a handlebody
$\Omega^{n+1}$, a diffeomorphism $\Phi:\partial\Omega\to M$, and an
immersion $G:\Omega\looparrowright\R^{n+1}$ such that
$G|_{\partial\Omega}=F\circ\Phi$.  Thus the result supplies the
genuine codimension-zero extension required in \cref{def:immersed-flat-filling},
not merely a diffeomorphism classification of the boundary.

\begin{proof}[Proof of \cref{thm:two-convex}]
By \cref{thm:HS}, the immersion admits an immersed filling.  The conclusion follows immediately from \cref{thm:flat-filling}.
\end{proof}

In dimension three, the implication from positive mean curvature and
positive scalar curvature to two-convexity is already noted in
\cite[Lemma~2.3 and the paragraph following it]{HuiskenSinestrari2009}.
For completeness, we record the short argument in our sign convention.

\begin{lemma}\label[lemma]{lem:3d-two-convex}
Let $F:M^3\looparrowright\R^4$ be a closed immersed
hypersurface with constant scalar curvature.  Choose the global normal
so that $H>0$, and order the principal curvatures by
\[
 \lambda_1\le\lambda_2\le\lambda_3.
\]
Then
\[
 \lambda_1+\lambda_2>0
\]
everywhere.
\end{lemma}

\begin{proof}
By \cref{lem:positivity}, $P_1=H\Id-A$ is positive definite.  Its
eigenvalue in a principal direction corresponding to $\lambda_3$ is
\[
 H-\lambda_3=\lambda_1+\lambda_2.
\]
This eigenvalue is positive, which proves the claim.
\end{proof}

\begin{proof}[\textbf{Proof of \cref{cor:dimension-three}}]
By \cref{lem:3d-two-convex}, the immersion is strictly two-convex.  Applying \cref{thm:two-convex} completes the proof.
\end{proof}

In dimension three the equality chain takes the particularly transparent form
\[
 4V
 =\int_M\frac{H}{\sigma_2}\,d\mu
 \ge3\int_M\frac1H\,d\mu
 \ge4V.
\]
Thus the traceless second fundamental form vanishes identically.

\section{Gromov fillings and high-dimensional consequences}
\label{sec:gromov}

We now use Gromov's \emph{eigenvalue-count $k$-convexity}.  In the
convention of this paper, a cooriented immersed hypersurface
$W^{N-1}\looparrowright\R^N$, with coorienting unit normal $\nu$, is
$k$-convex in this sense if at least $k$ eigenvalues of
$A_\nu=D\nu$ are nonnegative at every point.  This is an
eigenvalue-count condition, rather than a condition on a sum of the
smallest principal curvatures.  
Our sign convention agrees with Gromov's normalization: for the round
sphere his second fundamental form is positive, and the concentric
sphere of radius $r$ has principal curvatures
$r^{-1}$ \cite[pp.~15--16]{Gromov1991}.  Hence no
additional sign reversal is made when applying his filling theorem
below.

\begin{theorem}[Gromov's immersed filling theorem \cite{Gromov1991}]\label{thm:Gromov-filling}
Let $N\ge3$, and let
$W^{N-1}\looparrowright\R^N$ be a smooth closed cooriented immersed
hypersurface.  Suppose that $W$ is $k$-convex in the eigenvalue-count sense above and
\[
 k>\frac N2.
\]
Then there exist a compact smooth $N$-manifold $V$ with
$\partial V\simeq W$ and an immersion $\widetilde G:V\looparrowright\R^N$
whose restriction to the boundary is the given immersion.  In
particular, $W$ admits an immersed filling (and, if necessary, closed
components of $V$ may be discarded).
\end{theorem}

This is the filling theorem stated in
\cite[p.~23]{Gromov1991}; see pp.~23--24 there for the construction.
Gromov's argument slices the hypersurface by a generic family of
parallel hyperplanes and fills the resulting immersed sections
inductively; the inequality $k>N/2$ rules out the interior head-on
collision that obstructs extension.  For a later treatment of related
mean-convexity and filling questions, see \cite{Gromov2014}.  That
reference is included for context; the eigenvalue-count filling theorem
used here is the 1991 result stated in \cref{thm:Gromov-filling}.

\begin{remark} 
\label{rem:convexity-comparison}
The two filling criteria used in this paper overlap but are not the
same.  If $M^n\looparrowright\R^{n+1}$ is two-convex in the
Huisken--Sinestrari sense, then
\[
 \lambda_1+\lambda_2\ge0
 \quad\Longrightarrow\quad
 \lambda_2\ge0.
\]
Hence at least $n-1$ principal curvatures are nonnegative, so the
hypersurface is $(n-1)$-convex in Gromov's eigenvalue-count sense.  If
$n\ge4$, then
\[
 n-1>\frac{n+1}{2},
\]
and \cref{thm:Gromov-filling} already supplies an immersed filling.
For $n=3$, however,
\[
 n-1=2=\frac{n+1}{2},
\]
which is exactly the borderline excluded by Gromov's strict inequality
$k>N/2$.  Thus the Huisken--Sinestrari filling theorem is essential for
the unrestricted three-dimensional consequence \cref{cor:dimension-three}.
The converse implication fails in general: having many nonnegative
principal curvatures does not force the sum of the two smallest ones to
be nonnegative.
\end{remark}

\begin{proof}[\textbf{Proof of \cref{thm:negative-index}}]
By \cref{lem:positivity}, the normal bundle is trivial and we choose
the global normal so that $H>0$.  Put
\[
 r_0=\left\lfloor\frac{n-2}{2}\right\rfloor,
 \qquad
 k_0=n-r_0=\left\lfloor\frac{n+3}{2}\right\rfloor.
\]
The hypothesis \eqref{eq:negative-index-hypothesis} says that at least
$k_0$ principal curvatures are nonnegative at every point.  Moreover,
\[
 k_0>\frac{n+1}{2}.
\]
Thus the cooriented immersion is $k_0$-convex in Gromov's sense, and
\cref{thm:Gromov-filling}, applied with $N=n+1$, supplies an immersed
filling.  The conclusion follows from \cref{thm:flat-filling}.
\end{proof}

\begin{proof}[\textbf{Proof of \cref{cor:scalar-pinching}}]
Assume that \eqref{eq:negative-index-hypothesis} fails at a point
$p\in M$.  Write
\[
 \ell=\indm(A_p),
 \qquad
 s=\#\{i:\lambda_i(p)>0\}.
\]
Then
\[
 \ell\ge \left\lceil\frac{n-1}{2}\right\rceil,
 \qquad
 s\le n-\ell\le m.
\]
Since $H(p)>0$, at least one principal curvature is positive, so
$s\ge1$.  Since also $\ell\ge1$, the sum of the positive principal
curvatures,
\[
 P=\sum_{\lambda_i(p)>0}\lambda_i(p),
\]
satisfies $P>H(p)$: the omitted nonpositive principal curvatures have
strictly negative total sum.  Cauchy--Schwarz therefore gives
\[
 |A|^2(p)
 \ge \sum_{\lambda_i(p)>0}\lambda_i(p)^2
 \ge \frac{P^2}{s}
 >\frac{H^2(p)}{s}
 \ge\frac{H^2(p)}{m}.
\]
Using $2\sigma_2=H^2-|A|^2$, we obtain
\[
 2\sigma_2
 <\frac{m-1}{m}H^2(p),
\]
or equivalently
\[
 H^2(p)>\frac{2m}{m-1}\,\sigma_2,
\]
contrary to \eqref{eq:scalar-pinching}.  Hence
\eqref{eq:negative-index-hypothesis} holds everywhere, and
\cref{thm:negative-index} completes the proof.
\end{proof}

\begin{remark}\label[remark]{rem:high-dimensional-range}
The hypothesis in \cref{thm:negative-index} permits at most one
negative principal curvature when $n=4,5$, at most two when $n=6,7$,
and at most three when $n=8,9$.  As an eigenvalue-count hypothesis, it is weaker than two-convexity:
even when only one principal curvature is negative, no lower bound on
the sum of the two smallest curvatures is imposed; see
\cref{rem:convexity-comparison}.  The constants in
\cref{cor:scalar-pinching} begin with
\[
 \frac{2m}{m-1}
 =4\quad(n=4),
 \qquad
 =3\quad(n=5,6),
 \qquad
 =\frac83\quad(n=7,8).
\]
This is the threshold produced by the elementary implication from
the pinching condition to the inertia bound; no optimality claim for
the geometric sphere theorem is made here.
\end{remark}

\subsection{The remaining higher-dimensional obstruction}
\label{sec:obstruction}

The preceding results sharply constrain any possible nonround example in
dimension $n\ge4$.  For every closed constant-scalar-curvature immersion,
\cref{lem:positivity} gives
\[
 \sigma_2>0,
 \qquad H>0,
 \qquad P_1>0.
\]
If the immersion is two-convex, \cref{thm:two-convex} applies, while the
negative-inertia hypothesis \eqref{eq:negative-index-hypothesis} is covered
by \cref{thm:negative-index}.  Moreover,
\cref{cor:scalar-pinching} shows that the latter hypothesis follows from
the scalar pinching condition.  Consequently, any nonround counterexample
must contain a point $p$ at which
\begin{equation}\label{eq:necessary-bad-point}
 \indm(A_p)
 \ge \left\lceil\frac{n-1}{2}\right\rceil
 \quad\text{and}\quad
 H^2(p)>
 \frac{2\lfloor (n+1)/2\rfloor}
      {\lfloor (n+1)/2\rfloor-1}\,\sigma_2.
\end{equation}
Thus a counterexample would have to possess, somewhere, at least roughly
half as many negative as positive principal-curvature directions and at
the same time violate the explicit pinching threshold above.

These pointwise alternatives cannot be ruled out from the Gauss relation
alone.  Indeed, set
\[
 r=\left\lceil\frac{n-1}{2}\right\rceil,
 \qquad
 s=n-r=\left\lfloor\frac{n+1}{2}\right\rfloor,
\]
and consider the algebraic curvature vector
\[
 (\underbrace{-1,\ldots,-1}_{r\ \mathrm{entries}},
  \underbrace{L,\ldots,L}_{s\ \mathrm{entries}}).
\]
It satisfies
\[
 H=sL-r,
 \qquad
 \sigma_2=\binom r2-rsL+\binom s2L^2.
\]
For all sufficiently large $L$, one has $H>0$, $\sigma_2>0$, and
$H>|\lambda_i|$ for every $i$, so $P_1>0$, while
$\indm(A)=r$ violates \eqref{eq:negative-index-hypothesis}.  This is only
a pointwise algebraic model; it is not asserted to occur on a closed
constant-scalar-curvature immersion.  Its role is simply to show that the
basic admissibility consequences of \cref{lem:positivity} do not by
themselves settle the higher-dimensional problem.

The remaining issue also has a useful integral formulation.  Define
\[
 V_{\mathrm{alg}}
 =\frac1{n+1}\int_M u\,d\mu.
\]
By \cref{lem:minkowski},
\[
 V_{\mathrm{alg}}
 =\frac{n-1}{2(n+1)\sigma_2}\int_M H\,d\mu>0,
\]
so $V_{\mathrm{alg}}$ is independent of the choice of the base point
$a$.  The Minkowski identity together with \eqref{eq:newton-defect}
gives
\begin{equation}\label{eq:reverse-HK}
 \int_M\frac1H\,d\mu
 \le\frac{n+1}{n}V_{\mathrm{alg}}
 =\frac1n\int_M u\,d\mu,
\end{equation}
with equality if and only if the hypersurface is totally umbilical.  If
an immersed filling is available, \cref{lem:normal-compatibility}
identifies $V_{\mathrm{alg}}$ with the Riemannian volume of the flat
filling, and \cref{prop:HK} supplies the opposite inequality.  Hence the
unrestricted immersed problem would follow either from a more general
existence theorem for immersed fillings or from a filling-free (for
example, current-theoretic) proof of
\[
 \int_M\frac1H\,d\mu
 \ge\frac1n\int_M u\,d\mu.
\]
In this sense, the two-convexity and negative-inertia assumptions used
above serve only to solve the extension problem, through the theorems of
Huisken--Sinestrari and Gromov, respectively; once an immersed filling is
available, the rigidity step is the same Ros--Reilly argument.

\end{document}